\documentclass[12pt]{amsart}
\usepackage{amscd, amssymb, pgf, tikz, hyperref}

\newcommand{\field}[1]{\mathbb{#1}}

\newcommand{\TT}{\field{T}}

\newcommand{\la}{\langle}
\newcommand{\ra}{\rangle}

\newcommand{\ds}{\displaystyle}

\newtheorem{thm}{Theorem}[section]
\newtheorem{cor}[thm]{Corollary}

\newtheorem{prop}[thm]{Proposition}

\theoremstyle{definition}
\newtheorem{dfn}[thm]{Definition}

\theoremstyle{remark}
\newtheorem{rmk}[thm]{Remark}

\newtheorem{example}[thm]{Example}

\newtheorem*{examples*}{Examples}

\numberwithin{equation}{subsection}

\title[Group actions on graphs and crossed products]{Group actions on graphs and $C^*$-correspondences}

\author{Valentin Deaconu}
\address{Valentin Deaconu \\ Department of Math
 \& Stat (084)\\ University
of Nevada\\ Reno NV 89557-0084\\ USA} \email{vdeaconu@unr.edu}


\begin{document}
\begin{abstract}
 If $G$ acts on a  $C^*$-correspondence ${\mathcal H}$ over the $C^*$-algebra $A$ (see Definition \ref{act}), then by the universal property $G$ acts  on the Cuntz-Pimsner algebra ${\mathcal O}_{\mathcal H}$ and we study the crossed product  ${\mathcal O}_{\mathcal H}\rtimes G$ and the fixed point algebra ${\mathcal O}_{\mathcal H}^G$.  Using intertwiners, we define the Doplicher-Roberts algebra ${\mathcal O}_\rho$ of a representation $\rho$ of a compact group $G$ on ${\mathcal H}$ and prove that under certain conditions ${\mathcal O}_{\mathcal H}^G$ is isomorphic to ${\mathcal O}_\rho$. The action of $G$ commutes with the gauge action on ${\mathcal O}_{{\mathcal H}}$, therefore $G$ acts also on the core algebras ${\mathcal O}_{\mathcal H}^{\mathbb T}$, where $\mathbb T$ denotes the unit circle.
We give applications for the action of a group $G$ on  the $C^*$-correspondence ${\mathcal H}_E$ associated  to a topological graph $E$.
If $G$ is finite and  $E$ is discrete and locally finite, we prove that the crossed product $C^*(E)\rtimes G$ is isomorphic to the $C^*$-algebra of a graph of $C^*$-correspondences and stably isomorphic to a locally finite graph algebra. If $C^*(E)$ is simple and purely infinite and the action of $G$ is outer, then $C^*(E)^G$ and $C^*(E)\rtimes G$ are also simple and purely infinite with the same $K$-theory groups.  We illustrate with several examples.
\end{abstract}

\maketitle
\section{introduction}

Suppose the group $G$ acts on a directed (topological) graph $E$. This means that $G$ acts on the vertex space $E^0$ and on the edge space $E^1$, preserving incidences. By duality, we get an action of $G$ on the $C^*$-algebra $C_0(E^0)$ and on the space $C_c(E^1)$, which extends to the $C_0(E^0)-C_0(E^0)$ $C^*$-correspondence ${\mathcal H}_E$.
 In particular, there is a homomorphism $\rho:G\to {\mathcal L}_{\mathbb C}({\mathcal H}_E)$ into the set of invertible ${\mathbb C}$-linear operators on ${\mathcal H}_E$, called a representation  of $G$ on  ${\mathcal H}_E$. By the universal property, this  determines an action of $G$ on the Cuntz-Pimsner algebra ${\mathcal O}_{{\mathcal H}_E}$, also called the graph $C^*$-algebra and denoted $C^*(E)$.  For example, if $G$ finite acts on the graph with one vertex and $n$ loops, then
 we get an $n$-dimensional representation $\rho: G\to {\mathcal L}({\mathbb C}^n)$ and an action on the Cuntz algebra ${\mathcal O}_n$. It is known that  the fixed point algebra ${\mathcal O}_n^G$ is isomorphic to the Doplicher-Roberts algebra ${\mathcal O}_\rho$ (denoted by ${\mathcal O}_G$ in \cite {DR1}), which in turn is a full corner in a Cuntz-Krieger algebra (see \cite{MRS}).

 In a more general setting, given a group $G$ acting on a $A-A$ $C^*$-correspondence ${\mathcal H}$, our goal is to study the  fixed point algebra ${\mathcal O}_{\mathcal H}^G$ and the crossed product ${\mathcal O}_{\mathcal H}\rtimes G$.
 We define the Doplicher-Roberts algebra ${\mathcal O}_\rho$ associated to $\rho:G\to {\mathcal L}_{\mathbb C}({\mathcal H})$  from intertwiners $(\rho^m, \rho^n)$, where $\rho^n=\rho^{\otimes n}$ is the tensor power representation of $G$ on the balanced tensor product ${\mathcal H}^{\otimes n}$.
We prove that in certain cases ${\mathcal O}_\rho$ is isomorphic to ${\mathcal O}_{\mathcal H}^G$ and strongly Morita equivalent to
 ${\mathcal O}_{\mathcal H}\rtimes G$.

 If $G$ is finite and it acts on a  discrete and locally finite graph $E$, we prove that $C^*(E)\rtimes G$ is isomorphic to the $C^*$-algebra of a graph of (minimal) $C^*$-correspondences, constructed using the orbits in $E^0$ and $E^1$ and the characters of the stabilizer groups.  In the proof we use results about the crossed product of a $C^*$-correspondence by a group $G$.
As a consequence, $C^*(E)\rtimes G$ is strongly Morita equivalent to a graph algebra, so its $K$-theory can be computed in terms of the incidence matrix.
Since the action of $G$ commutes with the gauge action of $\TT$ on $C^*(E)$, the group $G$ also acts  on the core AF-algebra $C^*(E)^{\mathbb T}$ and $C^*(E)^{\mathbb T}\rtimes G\cong (C^*(E)\rtimes G)^{\mathbb T}$ is an AF-algebra. We  recover some examples of  group actions on AF-algebras considered by Handelman and Rossmann, see \cite{HR2}.

The paper is organized as follows. In the first section we define group actions on topological graphs and on $C^*$-correspondences, and we extend these actions to the associated Cuntz-Pimsner algebras.  In the next section we define the Doplicher-Roberts algebra associated to a group action on a $C^*$-correspondence.
We continue with general results about crossed products of $C^*$-correspondences and graphs of $C^*$-correspondences. The following section contains  the main result about finite group actions on discrete graphs. We conclude with several examples of group actions on graphs and $K$-theory computations for the crossed product and the fixed point algebra.

\subsection*{Acknowledgements} The author would like to thank Alex Kumjian and Bruce Blackadar for helpful and illuminating discussions.

\section{Group actions on graphs and graph algebras}

 Let $E=(E^0, E^1, r,s)$ be a topological graph (see \cite{DKQ, Ka04}). Recall that $E^0, E^1$ are locally compact Hausdorff spaces and $r,s :E^1\to E^0$ are continuous with $s$ a local homeomorphism. Denote by ${\mathcal H}={\mathcal H}_E$ its $C^*$-correspondence over $A=C_0(E^0)$, obtained by completing $C_c(E^1)$ with the inner product

\[\langle \xi,\eta\rangle(v)=\sum_{s(e)=v}\overline{\xi(e)}\eta(e),\; \xi,\eta\in C_c(E^1)\]
and multiplications

\[(\xi\cdot f)(e)=\xi(e)f(s(e)),\; (f\cdot\xi)(e)=f(r(e))\xi(e).\]
The $C^*$-algebra of a graph $E$ is defined as the Cuntz-Pimsner algebra ${\mathcal O}_{\mathcal H}$ of the $C^*$-correspondence ${\mathcal H}={\mathcal H}_E$, see \cite{Ka04}.
\begin{dfn} Let $E, F$ be two topological graphs. A graph morphism $\varphi:E\to F$ is a pair of continuous  maps $\varphi=(\varphi^0,\varphi^1)$ where $\varphi^i:E^i\to F^i, i=0,1$ such that $\varphi^0\circ r=r\circ \varphi^1$ and $\varphi^0\circ s=s\circ \varphi^1$, i.e. the diagram
\[\begin{CD}E^0@<s<<E^1@>r>>E^0\\@V\varphi^0VV@V\varphi^1VV@V\varphi^0VV\\F^0@<s<<F^1@>r>>F^0\end{CD}\]
is commutative.
An isomorphism of topological graphs is a graph morphism $\varphi=(\varphi^0,\varphi^1)$ such that $\varphi^i$ is
a homeomorphism for $i=0,1$. It follows that $\varphi^{-1}=((\varphi^0)^{-1},(\varphi^1)^{-1})$ is also a graph morphism. We denote by $Aut(E)$ the group of automorphisms of a topological graph $E$.
\end{dfn}
\begin{dfn} A locally compact group $G$ acts on  $E$ if there  are continuous maps $\alpha^i:G\times E^i\to E^i$, write $\alpha^i(g,x)=\alpha^i_g(x)$ for $i=0,1$ or just $g\cdot x$, such that $g\mapsto \alpha_g=(\alpha_g^0,\alpha_g^1)$ is a group homomorphism from $G$ into  $Aut(E)$. This means that $G$ acts on the vertex space $E^0$ and on the edge space $E^1$ such that the actions are compatible with the range and source maps $r,s$. This action can be extended to  finite paths $e_1\cdots e_n \in E^n$ by $g\cdot(e_1\cdots e_n)=(g\cdot e_1)\cdots (g\cdot e_n)$ and similarly to the set of infinite paths $E^\infty$.
\end{dfn}
\begin{rmk}
The group action on the graph $E$ determines a representation $\rho:G\to {\mathcal L}_{\mathbb C}({\mathcal H})$ by invertible ${\mathbb C}$-linear operators on  ${\mathcal H}$ and an
action of $G$ on $C_0(E^0)$ by $*$-automorphisms such that  \[ (\rho(g)\xi)(e)=\xi(g^{-1}\cdot e)\;\;  \text{for}\;  \xi\in C_c(E^1)\quad \text{and} \quad (g\cdot a)(v)=a(g^{-1}\cdot v)\;\; \text{for} \; a\in C_0(E^0).\] A routine verification shows that these actions are compatible with the inner product and the bimodule structure.
\end{rmk}
\begin{dfn}\label{act}
 We say that a locally compact group $G$ acts on the $C^*$-correspondence ${\mathcal H}$ over the $C^*$-algebra $A$ if $G$ acts  on ${\mathcal H}$ via a map $\rho :G\to {\mathcal L}_{\mathbb C}({\mathcal H})$ such that $\rho(g)$ is a ${\mathbb C}$-linear isomorphism and for all $\xi\in {\mathcal H}$ the map $g\mapsto \rho(g)\xi$ is norm continuous,  $G$ acts by $*$-automorphisms on $A$ such that  for all $a\in A$ the map $g\mapsto g\cdot a$ is norm continuous, and  the following compatibility relations are satisfied
 \[ \langle \rho(g)\xi,\rho(g)\eta\rangle=g\cdot\langle\xi,\eta\rangle ,\]\[ \rho(g)(\xi a)=(\rho(g)\xi)(g\cdot a),\; \rho(g)(a\xi)=(g\cdot a)(\rho(g)\xi).\]
 The map $\rho$   is called a representation of $G$ on $\mathcal H$.
 \end{dfn}

 \begin{rmk}
In particular, a group action on the graph $E$ determines as above a group action on the $C^*$-correspondence ${\mathcal H}_E$.  Notice though that a group  action on a  $C^*$-correspondence associated to a directed graph is not necessarily determined by an action on the  graph, see Example \ref{sym}.
 \end{rmk}

\begin{thm}
An action of $G$ on the $C^*$-correspondence ${\mathcal H}$ determines in a natural way an action on ${\mathcal K}_A({\mathcal H})$, the $C^*$-algebra generated by the finite rank operators, and an action on the Cuntz-Pimsner algebra ${\mathcal O}_{\mathcal H}$. The action of $G$ commutes with the gauge action, therefore we  get  an action of $G$ on the core algebra ${\mathcal O}_{\mathcal H}^{\mathbb T}$, the fixed point algebra under the gauge action. In particular, an action on a topological graph $E$ determines an action on the graph algebra $C^*(E)$ by $g\cdot S_e=S_{g\cdot e}$, where $S_e$ is a generator of $C^*(E)$ for $e\in E^1$, and an action on the core algebra $C^*(E)^{\mathbb T}$.
\end{thm}
\begin{proof} Recall that ${\mathcal K}_A({\mathcal H})$ is generated by operators $\theta_{\xi,\eta}$ where $\theta_{\xi,\eta}(\zeta)=\xi\la\eta,\zeta\ra$ and we define $g\cdot\theta_{\xi,\eta}=\theta_{\rho(g)\xi,\rho(g)\eta}$. The first part follows from the universal property of ${\mathcal O}_{\mathcal H}$. Recall that the gauge action $\gamma$ on ${\mathcal O}_{\mathcal H}$ is defined on generators by $\gamma(z)a=a, \gamma(z)\xi=z\xi$ for $z\in {\mathbb T}$ and is extended to ${\mathcal O}_{\mathcal H}$ using the universal property. Since $\rho(g):{\mathcal H}\to{\mathcal H}$ is ${\mathbb C}$-linear, we have $\rho(g)(z\xi)=z\rho(g)\xi$, so we get an action of $G$ on the core algebra ${\mathcal O}_{\mathcal H}^{\mathbb T}$.
\end{proof}
Recall that a discrete graph (a topological graph where $E^0, E^1$ are at most countable) is row finite if each vertex receives finitely many edges, and is locally finite if in addition each vertex emits finitely many edges. For free actions on discrete graphs we have the following result:

\begin{thm}(Kumjian and Pask, \cite{KP})
If $G, E$ are discrete, the action of $G$ on $E$ is free and $E$ is  locally finite,  then $C^*(E)^G\cong C^*(E/G)$ and
\[C^*(E)\rtimes G\cong C^*(E/G)\otimes {\mathcal K}(\ell^2(G)),\]
where $E/G$ is the quotient graph.
\end{thm}
This result is inspired from a theorem of Green about group actions on locally compact spaces, see \cite{G}.
A similar result was proved for free and proper actions of locally compact groups on topological graphs in \cite{DKQ}, namely that $C^*(E)\rtimes_rG$ is strongly Morita equivalent to $C^*(E/G)$.

In the same paper \cite{KP}, Kumjian and Pask  showed that if $G$ is abelian and  $c:E^1\to \hat{G}$ is  a cocycle, then this induces an action of $G$ on $C^*(E)$ such that $C^*(E)\rtimes G$ is isomorphic to $C^*(E(c))$, where  $E(c)$ is the skew product graph $(\hat{G}\times E^0,\hat{G}\times E^1, r,s)$ with
\[r(\chi,e)=(\chi c(e), r(e)), s(\chi,e)=(\chi,s(e))\]
for $\chi\in \hat{G}$.
By diagonalization, the action of $G$ on $C^*(E)$ is equivalent to the action $\alpha$  given by $\ds \alpha_g(S_e)=\la c(e),g\ra S_e$, where $S_e$ are the generators of $C^*(E)$.
\begin{rmk}
If $G$ abelian acts on the ${\mathcal O}_n$-graph with   $E^1=\{e_1,e_2,...,e_n\}$ and $E^0=\{v\}$, a cocycle $c:E^1\to \hat{G}$ determines a  representation $\rho$ of $G$ on ${\mathcal H}=span\{\xi_1,\xi_2,...,\xi_n\}$, where
$\rho(g)\xi_i=\la c(e_i),g\ra \xi_i$. Conversely, an  $n$-dimensional representation of the abelian group $G$ determines a cocycle on the ${\mathcal O}_n$-graph with values in $\hat{G}$.
\end{rmk}
\begin{rmk}
Actions of ${\mathbb Z}^l$ on $k$-graphs were studied by Farthing, Pask and Sims in \cite{FPS}. In particular, $K$-theory computations were done for actions of ${\mathbb Z}$ on a row finite $1$-graph with no sources such that  the orbit of each vertex is finite and either $K_0(C^*(E))$ or $K_1(C^*(E))$ is trivial.
\end{rmk}

\section{Doplicher-Roberts algebras}

The Doplicher-Roberts algebras (denoted  by ${\mathcal O}_G$ in \cite{DR1}) were introduced to construct a new duality theory for compact Lie groups $G\subseteq U(n)$ which strengthens the Tannaka-Krein duality. Let ${\mathcal T}_G$ denote the representation category whose objects are tensor powers of the  $n$-dimensional representation $\rho$ of $G$ defined by the inclusion $G\subseteq U(n)$ and whose arrows are the intertwiners. The  $C^*$-algebra ${\mathcal O}_G$ is identified in \cite{DR1} with the fixed point algebra ${\mathcal O}_n^G$, where ${\mathcal O}_n$ is the Cuntz algebra. If $\sigma_G$ denotes the restriction to ${\mathcal O}_G$ of the canonical endomorphism of the Cuntz algebra, then ${\mathcal T}_G$ can be reconstructed from the pair $({\mathcal O}_G,\sigma_G)$. Subsequently, Doplicher-Roberts algebras were associated to any object $\rho$ in a strict tensor $C^*$-category, see \cite {DR2}, \cite{DPZ}.

Suppose that the group $G$ acts on the $C^*$-correspondence ${\mathcal H}$ over $A$ via the representation $\rho:G\to {\mathcal L}_{\mathbb C}({\mathcal H})$. Inspired from \cite{DR1}, we consider the tensor power representation $\rho^n:G\to {\mathcal L}_{\mathbb C}({\mathcal H}^{\otimes n})$,  where ${\mathcal H}^{\otimes n}$ is the balanced tensor product of $n$ copies of ${\mathcal H}$ over $A$, and we define the set $(\rho^m,\rho^n)$ of intertwining operators by
\[(\rho^m,\rho^n)=\{T\in{\mathcal L}_A({\mathcal H}^{\otimes n},{\mathcal H}^{\otimes m})\mid   T\rho^n=\rho^mT\}.\]
 By definition ${\mathcal H}^{\otimes 0}=A$ and $\rho^0:G\to {\mathcal L}_{\mathbb C}(A)$ is the trivial representation $\rho^0(g)(a)=a$. We identify $(\rho^m,\rho^n)$ with a subset of $(\rho^{m+r},\rho^{n+r})$ via $T\mapsto T\otimes I_r$, where $I_r:{\mathcal H}^{\otimes r}\to {\mathcal H}^{\otimes r}$ is the identity map.
After this identification, it follows that the linear span ${}^0{\mathcal O}_\rho$ of $\displaystyle \bigcup_{m,n\ge 0}(\rho^m, \rho^n)$ has a natural multiplication  given by composition: if $S\in (\rho^m,\rho^n)$ and $T\in (\rho^p,\rho^q)$, then the product $ST$ is
 \[(S\otimes I_{p-n})\circ T\in (\rho^{m+p-n},\rho^q) \;\text{if}\; p\ge n,\]
 or
 \[S\circ(T\otimes I_{n-p})\in(\rho^m,\rho^{q+n-p}) \;\text{if}\; p<n.\]
 The adjoint of $T\in(\rho^m,\rho^n)$ is $T^*\in (\rho^n,\rho^m)$.
 We assume that
 \[\|T\|=\sup\{\|\pi(T)\|:\;\pi\;\text{is a}\;*-\text{representation of}\; {}^0{\mathcal O}_\rho\;\text{on a Hilbert space}\}\]
 is finite.
 \medskip
\begin{dfn}
 Under this assumption, we define the Doplicher-Roberts algebra ${\mathcal O}_\rho$ associated to the representation $\rho:G\to {\mathcal L}_{\mathbb C}({\mathcal H})$ as the $C^*$-closure of the normed $*$-algebra ${}^0{\mathcal O}_\rho$ with the above operations.
\end{dfn}
\begin{rmk} The $*$-algebra ${}^0{\mathcal O}_\rho$ has a natural ${\mathbb Z}$-grading and tensoring with $I$ on the left induces a $*$-endomorphism $\sigma$.
\end{rmk}

\begin{thm} Let ${\mathcal H}$ be a full finite projective $C^*$-correspondence over $A$ (i.e. ${\mathcal H}$ is a direct summand of $A^k$ for some $k$ and the inner products generate $A$) and assume that the left multiplication $A\to {\mathcal L}(\mathcal H)$ is injective. If $G$ is a compact group acting on $\mathcal H$ via $\rho:G\to  {\mathcal L}_{\mathbb C}({\mathcal H})$, then the Doplicher-Roberts algebra ${\mathcal O}_\rho$ is well defined and it is isomorphic to the fixed point algebra ${\mathcal O}_{\mathcal H}^G$.
\end{thm}
\begin{proof}
Since ${\mathcal H}$ is finite projective and the left multiplication is injective, it is known that ${\mathcal L}_A({\mathcal H})\cong{\mathcal K}_A({\mathcal H})$ and that ${\mathcal O}_{\mathcal H}$ is isomorphic to the $C^*$-algebra generated by the span of
$\ds \bigcup_{m,n\ge 0}{\mathcal K}_A({\mathcal H}^{\otimes m}, {\mathcal H}^{\otimes n})$
after we identify $T$ with $T\otimes I$ (see Proposition 2.5 in \cite{KPW}).
Note that $G$ acts on ${\mathcal K}_A({\mathcal H}^{\otimes n}, {\mathcal H}^{\otimes m})$ by $(g\cdot T)(\xi)=\rho^m(g)T(\rho^n(g^{-1})\xi)$ and the fixed point algebra is $(\rho^m, \rho^n)$. It follows that ${}^0{\mathcal O}_\rho\subseteq {\mathcal O}_{\mathcal H}$ and  that ${\mathcal O}_\rho$ is isomorphic to ${\mathcal O}_{\mathcal H}^G$.
\end{proof}
\begin{cor} Let $E$ be a topological graph such that ${\mathcal H}_E$ is full finite projective  and the left multiplication of $C_0(E^0)$ is injective. If $G$ is a compact
group acting on $E$,
and  $\rho:G\to {\mathcal L}_{\mathbb C}({\mathcal H}_E)$ denotes the representation,
then ${\mathcal O}_\rho\cong C^*(E)^G$.

Moreover, if  $C^*(E)$ is simple and purely infinite, $G$ is finite and the action on $C^*(E)$ is (pointwise) outer, then ${\mathcal O}_\rho$ and $C^*(E)\rtimes G$ are simple purely infinite and have the same $K$-theory, therefore are stably isomorphic.
 \end{cor}
 \begin{proof}

The first part follows directly from the above theorem. The second part  is a consequence of a result of A. Kishimoto and A. Kumjian (see Lemma 10 in \cite{KK} and Theorem 3.1 in \cite{K1}) : If $A$ is simple and purely infinite, $G$ is discrete and $\alpha:G\to Aut(A)$ is an action such that $\alpha_g$ is outer for all $g\in G\setminus\{e\}$, then $A\rtimes_{\alpha r}G$ is simple and purely infinite. The stable isomorphism follows from the fact that ${\mathcal O}_\rho\cong C^*(E)^G$ is a full corner in $C^*(E)\rtimes G$ and from classification results of simple separable purely infinite algebras satisfying UCT.
\end{proof}
\begin{rmk}
The natural inclusions $C^*(E)^G\subseteq C^*(E)\subseteq C^*(E)\rtimes G$ determine group homomorphisms \[K_0(C^*(E)^G)\to K_0(C^*(E))\to K_0(C^*(E)\rtimes G). \] Assuming $C^*(E)$ is unital, these homomorphisms give information on the class of the identity in $K_0(C^*(E)^G)$ and $K_0(C^*(E)\rtimes G)$.
\end{rmk}

 \begin{example}If $\gamma$ is the gauge action of ${\mathbb T}$ on a $C^*$-correspondence ${\mathcal H}$ over $A$, then ${\mathcal O}_\gamma\cong{\mathcal O}_{\mathcal H}^{\mathbb T}$.
 \end{example}
 \begin{example} If the group $G$ acts on a $C^*$-algebra $A$ and $\pi :G\to U(n)$ is a faithful unitary representation, then ${\mathcal H}={\mathbb C}^n\otimes A$ has a natural structure of $C^*$-correspondence over $A$ such that $G$ acts on ${\mathcal H}$ by $\rho(g)(x\otimes a)=\pi(g)x\otimes g\cdot a$. It is easy to check that in this case ${\mathcal O}_\rho$ is well defined and it is isomorphic to ${\mathcal O}_\pi\otimes A^G$, where ${\mathcal O}_\pi$ is the (old) Doplicher-Roberts algebra associated to the representation $\pi$.
 \end{example}


 \begin{rmk} Let the group $G$ act on the $C^*$-correspondence ${\mathcal H}$ over $A$. We have $A^G\subseteq (\rho, \rho)$, where $A^G$ denotes the fixed point algebra. Indeed, if $a\in A^G$, then \[a(\rho(g)\xi)=(g\cdot a)(\rho(g)\xi)=\rho(g)(a\xi).\]
\end{rmk}

\begin{example}
Consider a finite group $G$ acting  on the graph $E_n$ with one vertex and $n\ge 2$ edges. We denote by $\rho$ the corresponding representation on  ${\mathcal H}={\mathcal H}_E={\mathbb C}^n$.

Let $\hat{G}$ denote the set of equivalence classes of irreducible unitary representations, and construct as in \cite{MRS} a graph with the incidence matrix $B=B(\rho)$, where $B(v,w)$ is the multiplicity of $w$ in $v\otimes \rho$ for $v,w\in \hat{G}$.
It is shown in \cite{MRS}  that ${\mathcal O}_\rho$ is a full corner in the Cuntz-Krieger algebra ${\mathcal O}_B$.

For $G=S_n$ the symmetric group acting by permuting the edges of $E_n$, we get an outer  action on the Cuntz algebra ${\mathcal O}_n$ such that ${\mathcal O}_n\rtimes S_n$ is simple and purely infinite, stably isomorphic to
 ${\mathcal O}_\rho\cong {\mathcal O}_n^{S_n}$.  We also get an action of $S_n$ on the core algebra ${\mathcal O}_n^{\mathbb T}\cong M_{n^\infty}$ such that $M_{n^\infty}\rtimes S_n$ is AF.

For $n=3$, using the character table of $S_3$, it was calculated in \cite{MRS} that \[B=\left[\begin{array}{ccc}1&0&1\\0&1&1\\1&1&2\end{array}\right],\] which gives
\[K_0({\mathcal O}_3\rtimes S_3)= K_0({\mathcal O}_\rho)= K_0({\mathcal O}_B)\cong{\mathbb Z},\]\[ K_1({\mathcal O}_3\rtimes S_3)= K_1({\mathcal O}_\rho)= K_1({\mathcal O}_B)\cong{\mathbb Z},\]
\[K_0(M_{3^\infty}\rtimes S_3)\cong \varinjlim({\mathbb Z}^3,B).\]

The inclusions ${\mathcal O}_3^{S_3}\hookrightarrow {\mathcal O}_3\hookrightarrow {\mathcal O}_3\rtimes S_3$ determine the $K_0$-theory maps ${\mathbb Z}\to {\mathbb Z}_2\to {\mathbb Z}$. In particular the action of $S_3$ on ${\mathcal O}_3$ does not have the Rokhlin property, since the  map ${\mathbb Z}\to {\mathbb Z}_2$ is not injective (see Theorem 3.13 in \cite{Iz})
and ${\mathcal O}_3^{S_3}$, ${\mathcal O}_3\rtimes S_3$ are not isomorphic since  the classes of the identity in their $K_0$-groups do not coincide.

\end{example}
\begin{rmk}
If $R(S_3)\cong K_0(S_3)$ is the representation ring of $S_3$, then the matrix $B$ above is determined by the map $R(S_3)\to R(S_3)$ given by multiplication with the character of the representation $\rho$ (see \cite{HR1, HR2}).
\end{rmk}
\begin{rmk}
An action of a group $G$ on a row-finite (discrete) graph $E$ with no sources induces an action of $G$ on the associated graph groupoid ${\mathcal G}={\mathcal G}_E$ such that $C^*(E)\rtimes G\cong C^*({\mathcal G}\rtimes G)$, where ${\mathcal G}\rtimes G$ is the semidirect product groupoid with multiplication
\[(\gamma, g)(g^{-1}\cdot\gamma',g')=(\gamma\gamma', gg'),\] inverse operation \[(\gamma, g)^{-1}=(g^{-1}\cdot\gamma^{-1}, g^{-1})\] and range and source maps \[r(\gamma, g)=(r(\gamma),e), \;\; s(\gamma, g)=(g^{-1}\cdot s(\gamma), e).\] The unit space of ${\mathcal G}\rtimes G$ is identified with ${\mathcal G}^0$.

In particular, for the $S_n$  action above we get an action of $S_n$
on the Cuntz groupoid \[{\mathcal G}_n=\{(x,p-q,y)\in X\times{\mathbb Z}\times X: \sigma^px=\sigma^qy\},\] where $\sigma:X\to X$ is the shift
on the unit space $X=\{1,...,n\}^{\mathbb N}$ such that \[{\mathcal O}_n\rtimes S_n\cong C^*({\mathcal G}_n\rtimes S_n).\]
\end{rmk}
\begin{example}
Given a finite-dimensional unitary representation $\rho$ of a compact group $G$, Kumjian, Pask, Raeburn and Renault (see \cite{KPRR})  realize the Doplicher-Roberts algebra ${\mathcal O}_\rho$ as a corner in a graph $C^*$-algebra  and as a groupoid algebra. The graph has vertices $\hat{G}$, the set of equivalence classes of irreducible representations, and the groupoid is the reduction of the graph groupoid to the set of infinite paths starting at the trivial representation. It turns out that if $\rho$ takes values in $SU(n)$ and is faithful, then the graph is irreducible and locally finite, in particular ${\mathcal O}_\rho$ is simple.
Moreover, if $n\ge 2$, $G$ is an infinite compact Lie group and $\beta_\rho$ denotes the endomorphism of the representation ring $R(G)$ given by tensoring with $\rho$, then $K_0({\mathcal O}_\rho)\cong R(G)/\text{im} (1-\beta_\rho)$ and $K_1({\mathcal O}_\rho)=0$. This last result appeared also in A. Wassermann's thesis \cite{W}.
\end{example}

\section{Group actions on $C^*$-correspondences and crossed products}

We will need to allow $B$--$A$  $C^*$-correspondences where $A$ and $B$ are not necessarily the same $C^*$-algebras, so we  extend our notion of group action:

\begin{dfn}
Given  $C^*$-algebras $A,B$ and a $B$--$A$ $C^*$-correspondence ${\mathcal H}$, an action of a locally compact group $G$ on ${\mathcal H}$ is determined by a homomorphism $\rho:G\to {\mathcal L}_{\mathbb C}({\mathcal H})$ such that $\rho(g)$ is a ${\mathbb C}$-linear isomorphism and $g\mapsto \rho(g)\xi$ is continuous  and continuous actions of $G$ on $A$ and $B$ by $*$-automorphisms with compatibility relations
\[ \langle \rho(g)\xi,\rho(g)\eta\rangle=g\cdot\langle\xi,\eta\rangle ,\]\[ \rho(g)(\xi a)=(\rho(g)\xi)(g\cdot a),\; \rho(g)(b\xi)=(g\cdot b)(\rho(g)\xi),\]
where $\xi\in {\mathcal H}, a\in A, b\in B$.
\end{dfn}
As we mentioned before, an action of $G$ on ${\mathcal H}$ determines an action of $G$ on ${\mathcal K}(\mathcal H)$ given by $g\cdot \theta_{\xi,\eta}=\theta_{\rho(g)\xi, \rho(g)\eta}$, where $\theta_{\xi,\eta}(\zeta)=\xi\la \eta,\zeta\ra$.
Recall that if $A=B$, an action of $G$ on ${\mathcal H}$  determines an action on the Cuntz-Pimsner algebra ${\mathcal O}_{\mathcal H}$ (called quasi-free) and, since the action commutes with the gauge action, an action on the core algebra ${\mathcal O}_{\mathcal H}^{\mathbb T}$.

\begin{dfn}Suppose the group $G$ acts on the $B$--$A$ $C^*$-correspondence ${\mathcal H}$.
 The  crossed product $C^*$-correspondence ${\mathcal H}\rtimes G$ is defined
as ${\mathcal H}\rtimes G={\mathcal H}\otimes_{\varphi}(A\rtimes G)$, where $\varphi: A\to {\mathcal L}(A\rtimes G)$ is the embedding of $A$ in the multiplier algebra of $A\rtimes G$, regarded as a Hilbert module over itself.
 \end{dfn}
 \begin{rmk}
The crossed product ${\mathcal H}\rtimes G$ becomes a $B\rtimes G$--$A\rtimes G$ $C^*$-correspondence  after the completion of $C_c(G,{\mathcal H})$ using the operations
  \[\langle \xi, \eta\rangle(t)=\int_Gs^{-1}\cdot \langle \xi(s),\eta(st)\rangle ds,\]\[(\xi\cdot f)(t)=\int_G\xi(s)(s\cdot(f(s^{-1}t)))ds,\]\[(h\cdot\xi)(t)=\int_Gh(s)\cdot(s\cdot \xi(s^{-1}t))ds,\]
  where $\xi,\eta\in C_c(G,{\mathcal H}), f\in C_c(G,A), h\in C_c(G,B)$. Note that the right and left multiplications are given by convolution, and the inner product formula could be also expressed as
  \[\langle\xi\otimes f, \eta\otimes f'\rangle=f^*\langle \xi, \eta\rangle f',\]
  where this time $\xi, \eta\in {\mathcal H}$, $ f,f'\in C_c(G,A)$ and $f^*(t)=t\cdot f(t^{-1})^*$.
\end{rmk}
\begin{thm}(G. Hao, C.-K. Ng, \cite{HN}) Let ${\mathcal H}$ be a $C^*$-correspondence over $A$ and let the amenable locally compact group $G$ act on ${\mathcal H}$.  Then
\[{\mathcal O}_{{\mathcal H}\rtimes G}\cong {\mathcal O}_{\mathcal H}\rtimes G.\]
\end{thm}
\begin{cor} For $G$ amenable acting on a $C^*$-correspondence ${\mathcal H}$  we have
\[{\mathcal O}_{{\mathcal H}\rtimes G}^{\mathbb T}\cong {\mathcal O}_{\mathcal H}^{\mathbb T}\rtimes G.\]
\end{cor}
\begin{cor}
Let $G$ be a compact group acting on a topological graph $E$ with $C^*$-correspondence ${\mathcal H}_E$. Then
\[C^*(E)\rtimes G\cong {\mathcal O}_{{\mathcal H}_E\rtimes G}.\]
\end{cor}
\begin{example}Let $G$ be a compact group and let ${\mathcal E}$ be a Hermitian vector bundle over a locally compact space $X$ such that $G$ acts on both ${\mathcal E}$ and $X$ in a compatible way (see \cite{A}, section 1.6). Such a vector bundle is called a $G$-vector bundle, generalizing both ordinary vector bundles (when $G$ is trivial) and $G$-modules (when $X$ reduces to a point). The set of sections $\Gamma({\mathcal E})$ becomes in the usual way a  $C^*$-correspondence over $C_0(X)$, and the group $G$ acts on $\Gamma({\mathcal E})$. In particular, $G$ acts on its Cuntz-Pimsner algebra, which is a continuous field of Cuntz algebras, see \cite{V}.
\end{example}
\begin{example} Let $G$ be compact and let $\rho:G\to U(n)$ be a unitary representation. This determines an action of $G$ on  ${\mathcal O}_{\mathcal H}\cong {\mathcal O}_n$ where ${\mathcal H}={\mathbb C}^n$ and a product type action $\ds  \bigotimes _1^\infty Ad \rho$ on $\ds {\mathcal O}_{\mathcal H}^{\mathbb T}\cong \ds \bigotimes_1^\infty M_n\cong M_{n^\infty}$(see \cite{HR1}). We obtain the isomorphisms
\[{\mathcal O}_{{\mathcal H}\rtimes G}\cong {\mathcal O}_n\rtimes G,\;\; {\mathcal O}_{{\mathcal H}\rtimes G}^{\mathbb T}\cong M_{n^\infty}\rtimes G.\]
\end{example}
\begin{rmk}
Let $G$ be a compact group and let  $E$ be a finite graph with $C^*$-correspondence ${\mathcal H}_E$. If $G$ acts on ${\mathcal H}_E$, using the universal property we obtain an action of $G$ on $C^*(E)$. Since the action  commutes with the gauge action, we get  an action on the core algebra $C^*(E)^{\mathbb T}\cong \varinjlim A_n$, where  $A_n$ have dimension $m_n$. For a locally representable action as in  \cite{HR1, HR2}, $K_0(C^*(E)^{\mathbb T}\rtimes G)$ is the inductive limit of $K_0(A_n\rtimes G)\cong K_0(G)^{m_n}$, where the inclusion maps are determined by matrices with entries in the representation ring $K_0(G)\cong R(G)$.
\end{rmk}
\begin{rmk}
Note that some actions which permute vertices in a graph with more than one vertex may not induce locally representable actions on the core algebra (see Example \ref{toeplitz}).
\end{rmk}

\section{Graphs of $C^*$-correspondences and applications to finite groups actions on discrete graphs}

Given a discrete graph $E=(E^0,E^1,r,s)$, associate to each vertex $v\in E^0$ a $C^*$-algebra $A_v$ and to each edge $e\in E^1$ a nondegenerate $A_{r(e)}$--$A_{s(e)}$ $C^*$-correspondence ${\mathcal H}_e$. This way we obtain an $E$-system   of $C^*$-correspondences or a graph of $C^*$-correspondences. The $C^*$-algebra associated to this graph of $C^*$-correspondences is ${\mathcal O}_{\mathcal H}$, where $ {\mathcal H}=\bigoplus_{e\in E^1}{\mathcal H}_e$ becomes a $C^*$-correspondence over $A=\bigoplus_{v\in E^0}A_v$ in a natural way.  For more information, see \cite{DKPS}, where we discuss systems of $C^*$-correspondences over $k$-graphs $\Lambda$ and we construct a Fell bundle over the path groupoid ${\mathcal G}_\Lambda$ such that its reduced cross-sectional algebra is isomorphic to the $C^*$-algebra of the $\Lambda$-system. Unlike in \cite {DKPS}, here we allow graphs with sources and $C^*$-correspondences which are not full.

\begin{example}
Given a discrete graph $E$, associate to each vertex the $C^*$-algebra ${\mathbb C}$ and to each edge the $C^*$-correspondence ${\mathbb C}$. This is a graph of $C^*$-correspondences with associated $C^*$-algebra isomorphic to $C^*(E)$.
\end{example}

\begin{example}
Let $E$ have one vertex $v$ and one loop $e$, and let $A_v={\mathbb C}$, ${\mathcal H}_e={\mathbb C}^n$. Then the $C^*$-algebra of this graph of $C^*$-correspondences is ${\mathcal O}_n$. If $A_v=A$ is any $C^*$-algebra and ${\mathcal H}_e={\mathcal H}$ is a $C^*$-correspondence over $A$, then we get ${\mathcal O}_{\mathcal H}$.
\end{example}

\begin{example}
Consider a $C^*$-correspondence ${\mathcal H}$ over a unital $C^*$-algebra $A$ such that $A$ decomposes into a direct sum $A_1\oplus A_2\oplus\cdots \oplus A_n$.
If $p_j$ is the identity of $A_j$, then ${\mathcal H}$ decomposes into a direct sum $\ds\bigoplus_{i,j} p_i{\mathcal H}p_j$ and we can construct a graph of $C^*$-correspondences with $n$ vertices $\{v_1,...,v_n\}$, by assigning  the $C^*$-algebra $A_i$ at $v_i$ and the $A_i$--$A_j$ $C^*$-correspondence $p_i{\mathcal H}p_j\neq 0$ at an edge joining $v_j$ with $v_i$. If some of these correspondences are trivial, there is no edge between the corresponding vertices.
\end{example}
Recall that if a finite group $G$ acts on a finite or countable set $X$, then $C_0(X)\rtimes G$ decomposes as a direct sum of crossed products $C(Gx)\rtimes G$ over the orbit space $X/G$. Since the action on each orbit $Gx$ is transitive, this orbit can be identified with the homogeneous space $G/G_x$, where $G_x$ is the stabilizer group and $G$ acts on $G/G_x$ by left translation. Moreover, \[C(G/G_x)\rtimes G\cong M_{|Gx|}\otimes C^*(G_x),\]  which is isomorphic to a finite direct sum of matrix algebras.
\begin{cor} If a finite group $G$ acts on a discrete graph $E$, then
\[C_0(E^0)\rtimes G\cong \bigoplus_{E^0/G} M(Gv),\]
where $M(Gv)$ is a finite direct sum of matrix algebras. In particular, $C_0(E^0)\rtimes G$ is strongly Morita equivalent (SME) to a direct sum of finite dimensional abelian $C^*$-algebras.
\end{cor}
To describe the crossed product $C^*(E)\rtimes G$, we first consider the case when $C_0(E^0)\rtimes G$ is abelian.

\begin{prop} Suppose $C_0(E^0)\rtimes G=C_0(V)$ with $V$ finite or countable, and denote by $\{p_t\}_{t\in V}$ the minimal projections in $C_0(V)$.
The isomorphism classes of separable nondegenerate $C^*$-correspondences ${\mathcal H}$ over $C_0(V)$ with  $*$-homomorphism $\varphi: C_0(V)\to {\mathcal L}({\mathcal H})$ correspond to matrices $(a_{st})_{s,t\in V}$ where $a_{st}$ are nonnegative integer entries or $a_{st}=\infty$. More precisely, $a_{st}=\dim \varphi(p_s){\mathcal H}p_t.$
\end{prop}
\begin{proof}See Theorem 1.1 in \cite{KPQ}.
\end{proof}
\begin{cor} If $C_0(E^0)\rtimes G=C_0(V)$ is abelian, then $C^*(E)\rtimes G$ is the graph algebra with incidence  matrix $(a_{st})_{s,t\in V}$.
It is also the $C^*$-algebra of a graph of $C^*$-correspondences where the vertex algebras are ${\mathbb C}$ (one for each vertex $t\in V$) and the $C^*$-correspondences are  Hilbert spaces of dimension $a_{st}$.
\end{cor}
\begin{prop} Suppose $A$ and $B$ are SME $C^*$-algebras with $A$--$B$ imprimitivity bimodule ${\mathcal X}$. If ${\mathcal H}$ is a $C^*$-correspondence over $A$, then ${\mathcal H'}={\mathcal X}^*\otimes_A{\mathcal H}\otimes _A{\mathcal X}$ is a $C^*$-correspondence over $B$ such that ${\mathcal O}_{\mathcal H}$ and ${\mathcal O}_{\mathcal H'}$ are SME.
\end{prop}
\begin{proof} Let ${\mathcal R}={\mathcal H}\otimes _A{\mathcal X}$ and let ${\mathcal S}={\mathcal X}^*$. Then ${\mathcal R}\otimes_B{\mathcal S}\cong {\mathcal H}, {\mathcal S}\otimes_A{\mathcal R}\cong {\mathcal H'}$, so by a theorem in \cite{MPT} (see also \cite{MS}), we get that ${\mathcal O}_{\mathcal H}$ and ${\mathcal O}_{\mathcal H'}$ are SME.
\end{proof}

\begin{cor}
 Given a discrete locally finite graph $E=(E^0, E^1, r,s)$ and a finite group $G$ acting on $E$, the crossed product $C^*(E)\rtimes G$ is  SME to a locally finite graph $C^*$-algebra, where the number of vertices is the cardinality of the spectrum of $C_0(E^0)\rtimes G$.
 In particular, the $K$-theory of $C^*(E)\rtimes G$ and of $C^*(E)^{\mathbb T}\rtimes G$ can be computed if we determine the incidence matrix of the graph.
 \end{cor}
 \begin{proof}
 We apply the Proposition with $A=C_0(E^0)\rtimes G$ and $B=C_0(V)$.
 \end{proof}
\begin{thm}
Given a discrete locally finite graph $E=(E^0, E^1, r,s)$ and a finite group $G$ acting on $E$, the crossed product $C^*(E)\rtimes G$ is isomorphic to the $C^*$-algebra of a  graph of (minimal) $C^*$-correspondences, where at each vertex $v$ we associate a matrix algebra $M_{n(v)}$ and at each edge joining $w$ and $v$ we associate $M_{n(v),n(w)}$, the set of rectangular matrices with $n(v)$ rows and $n(w)$ columns.
\end{thm}
\begin{proof}
Since the group is finite, the orbits in $E^0$ and $E^1$ are finite. We  decompose the  $C^*$-correspondence ${\mathcal H}_E\rtimes G$ over the  $C^*$-algebra $C_0(E^0)\rtimes G$.
 This decomposition is obtained in two stages, from the orbits in $E^0$ and from the characters of the stabilizer groups.
For the first stage,
we consider the quotient graph $E/G$ and at each vertex $[v]\in E^0/G$  we associate the $C^*$-algebra $C(Gv)\rtimes G$, where  $Gv$ is the  orbit of $v$ in $E^0$ and at each edge $[e]\in E^1/G$ we associate the $C(Gr(e))\rtimes G$--$C(Gs(e))\rtimes G$ $C^*$-correspondence $C(Ge)\rtimes G$ of the  orbit $Ge$ in $E^1$.
For the second stage, we decompose each $C(Gv)\rtimes G\cong M_{|Gv|}\otimes C^*(G_v)$ into simple components.

Let $\ds C_0(E^0)\rtimes G\cong \bigoplus_{i=1}^mM_{n(i)}$, where $m\in {\mathbb N}\cup\{\infty\}$ and $M_{n(i)}$ denotes the set of $n(i)\times n(i)$  matrix algebras. Consider now the  graph with  $m$ vertices   and at each vertex $v_i$ we assign the $C^*$-algebra $M_{n(i)}$. If $p_i$ is the unit in $M_{n(i)}$, whenever $p_i({\mathcal H}_E\rtimes G)p_j\neq 0$, we decompose this as a direct sum of minimal $M_{n(i)}$--$M_{n(j)}$ $C^*$-correspondences. A minimal $C^*$-correspondence  is of the form $M_{n(i),n(j)}$, the set of rectangular matrices with $n(i)$ rows and $n(j)$ columns, with the obvious bimodule structure and inner product. Of course, $M_{n,n}=M_n$ and $M_{n,1}={\mathbb C}^n$. This decomposition determines  the number of edges between $v_j$ and $v_i$.
By construction, it follows that $C^*(E)\rtimes G$ is isomorphic to the $C^*$-algebra of this graph of $C^*$-correspondences.
\end{proof}

 \begin{rmk}
Given a compact group $G$, denote by $R(G)$ its representation ring. If $G$ acts on a $C^*$-algebra $A$, recall that $K_0(A\rtimes G)$ has a structure of $R(G)$-module. Indeed, given $M$  a finite dimensional $G$-module with character $\chi$ and $N$ a finitely generated projective $A\rtimes G$-module, then $M\otimes N$ has a structure of $A\rtimes G$-module and we can define the product $[N]\cdot\chi$ as $[M\otimes N]$.
In particular, given a finite group $G$ acting on a finite graph $E$, the groups $K_0(C^*(E)\rtimes G)$ and $K_0(C^*(E)^{\mathbb T}\rtimes G)$ have a structure of $R(G)$-modules.

 \end{rmk}
 \begin{rmk}
Nekrashevych (see \cite{N}) studied faithful actions of discrete groups $G$ on the set of finite words $X^*=\bigcup_{k=0}^\infty X^k$ over a finite alphabet $X$, which are  self-similar in the sense that for all $g\in G$ and $x\in X$ there exist  unique $y\in X$ and $h\in G$ such that $g\cdot(xw)=y(h\cdot w)$ where $w\in X^*$. A self-similar action determines an action of $G$ on the rooted tree $T_X$ with root $\emptyset$ and edges from $w\in X^*$ to $wx$ for $x\in X$. He constructed a $C^*$-correspondence $M=\bigoplus_{x\in X}C^*(G)$ over the  $C^*$-algebra $C^*(G)$, where the left action is the integrated form of a unitary representation of $G$ in ${\mathcal L}(M)$, defined using the self-similarity condition. Since in our notion of representation $\rho :G\to {\mathcal L}_{\mathbb C}(\mathcal H)$ the operator $\rho(g)$ is not $A$-linear, the $C^*$-correspondence $M$ is not the $C^*$-correspondence associated with the (infinite) tree $T_X$, and the Cuntz-Pimsner algebra ${\mathcal O}_M$, denoted also ${\mathcal O}_{(G,X)}$ in \cite{N},  is not isomorphic to the crossed product $C^*(T_X)\rtimes G$. Observe though that ${\mathcal O}_M$ contains a copy of the Cuntz algebra  ${\mathcal O}_n$, the tree $T_X$ is the universal covering of the graph $E_n$ with one vertex and $n=|X|$ edges, and $M={\mathbb C}^n\otimes C^*(G)$.

Exel and Pardo in \cite{EP}, inspired from  the Nekrashevych construction,  associate a $C^*$-algebra  ${\mathcal O}_{G,E}$  from a countable discrete group $G$ acting on a finite graph $E$ and a one-cocycle $\varphi:G\times E^1\to G$ which determines an action of $G$ on the space of finite paths $E^*$ such that $g\cdot (\alpha\beta)=(g\cdot \alpha)(\varphi(g,\alpha)\cdot\beta)$. This $C^*$-algebra is defined as the Cuntz-Pimsner algebra of a $C^*$-correspondence $M$ over $C(E^0)\rtimes G$ and contains a copy of the graph algebra $C^*(E)$. In particular for $G=\mathbb Z$ acting on a graph $E$ with $N\times N$ incidence matrix  $A$ by fixing the vertices  and permuting the edges in a way determined by another $N\times N$ integer matrix $B$ (which also determines the cocycle), they prove that  ${\mathcal O}_{G,E}$ is isomorphic to  Katsura's  algebra ${\mathcal O}_{A,B}$, see \cite{Ka08}, used to  model all Kirchberg algebras. Note that in this case $C(E^0)\rtimes G$ is isomorphic to the direct sum of $N$ copies of $C(\mathbb T)\cong C^*(\mathbb Z)$.

The relationship between the $C^*$-algebras studied by these authors and our crossed products $C^*(E)\rtimes G$ remains to be explored.
\end{rmk}

 \section{Examples}

\begin{example}
Let $G={\mathbb Z}_n$ act freely on its cyclic Cayley graph $E$ with $n$ vertices and $n$ edges.
We get a representation $\rho$ of ${\mathbb Z}_n$ on ${\mathcal H}={\mathbb C}^n$ and an action of ${\mathbb Z}_n$ on $A={\mathbb C}^n$ which permutes cyclically the basis.
In this case ${\mathcal L}_A({\mathcal H})=\{T\in M_n: T(\xi a)=T(\xi)a\}\cong{\mathbb C}^n$ (diagonal matrices).
Moreover, $(\rho, \rho)\cong{\mathbb C}I$ and ${\mathcal O}_\rho\cong C({\mathbb T})$, since the quotient graph $E/G$ has one vertex and one loop.

The crossed product ${\mathcal H}\rtimes{\mathbb Z}_n\cong M_n$ becomes a $C^*$-correspondence over $A\rtimes G\cong M_n$, and $C^*(E)\rtimes G\cong M_n\otimes C({\mathbb T})$. The graph of $C^*$-correspondences  has one vertex and one loop with $M_n$ attached to each.
\end{example}

\begin{example}
Let $E$ be the  graph with three vertices $v_1,v_2, v_3$ and edges connecting each $v_i$ with $v_j$ for $i\neq j$.

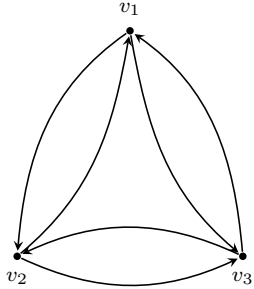
\[
\begin{tikzpicture}[shorten >=0.4pt, >=stealth, semithick]
\renewcommand{\ss}{\scriptstyle}
\node[inner sep=1.0pt, circle, fill=black]  (u) at (0,2) {};
\node[above] at (u.north)  {$\ss v_1$};
\node[inner sep=1.0pt, circle, fill=black]  (v) at (-1.5,-1) {};
\node[below] at (v.south)  {$\ss v_2$};
\node[inner sep=1.0pt, circle, fill=black]  (w) at (1.5,-1) {};
\node[below] at (w.south)  {$\ss v_3$};
\draw[->] (u) to [out=215, in=85] (v);{}
\draw[->] (v) to [out=40, in=-100] (u);
\draw[->] (w) to [out=95, in=-35] (u);
\draw[->] (u) to [out=-80, in=140] (w);
\draw[->] (v) to [out=-25, in=-155]  (w);
\draw[->] (w) to [out=155, in=25]  (v);

\end{tikzpicture}
\]
The permutation group $G=S_3$ acts (non-freely) on $E$ by permuting the vertices. The action on edges is uniquely determined. There is a single orbit in $E^0$ and $E^1$, and the stabilizer group for each vertex is ${\mathbb Z}_2$. Since ${\mathcal H}_E\cong {\mathbb C}^6$, the representation $\rho$ of $S_3$ is $6$-dimensional and ${\mathcal O}_\rho\cong C^*(E)^G\cong {\mathcal O}_2$. We have $C^*(E)\cong M_3({\mathcal O}_2)\cong {\mathcal O}_2$, so we get an action of $S_3$ on ${\mathcal O}_2$. Here $A={\mathbb C}^3$ and $A\rtimes S_3\cong M_3(C^*({\mathbb Z}_2))\cong M_3\oplus M_3$.
The $C^*$-correspondence ${\mathcal H}_E\rtimes G$ over $M_3\oplus M_3$ decomposes into $M_3\oplus M_3\oplus M_3\oplus M_3$. The graph of $C^*$-correspondences is
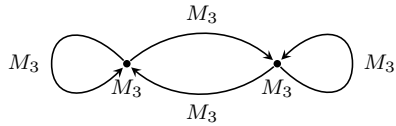
\[
\begin{tikzpicture}[shorten >=0.4pt,>=stealth, semithick]
\renewcommand{\ss}{\scriptstyle}
\node[inner sep=1.0pt, circle, fill=black]  (u) at (0,0) {};
\node[below] at (u.south)  {$\ss M_3$};
\node[inner sep=1.0pt, circle, fill=black]  (v) at (2,0) {};
\node[below] at (v.south)  {$\ss M_3$};
\draw[->] (u) to [out=40, in=140] node[above,black] {$\ss M_3$} (v);
\draw[->] (v) to [out=-140, in=-40] node[below,black] {$\ss M_3$} (u);
\draw[->] (v) .. controls (2.5,-0.5) and (3, -0.5) ..  (3,0)
node[right,black] {$\ss M_3$}
                               .. controls (3,0.5) and (2.5, 0.5) .. (v);
\draw[->] (u) .. controls (-0.5,0.5) and (-1, 0.5) ..  (-1,0)
node[left,black] {$\ss M_3$}
                               .. controls (-1,-0.5) and (-0.5, -0.5) ..
(u);
\end{tikzpicture}
\]
and $C^*(E)\rtimes G$ is stably isomorphic to ${\mathcal O}_2$.
\end{example}
\begin{example}\label{toeplitz}
Consider the graph $E$ with three vertices $v,v_1,v_2$ and four edges $e_1,e_2, f_1,f_2$ like in the figure.

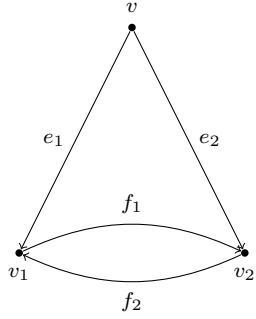
\[
\begin{tikzpicture}shorten >=0.4pt, >=stealth, semithick]
\renewcommand{\ss}{\scriptstyle}
\node[inner sep=1.0pt, circle, fill=black]  (u) at (0,2) {};
\node[above] at (u.north)  {$\ss v$};
\node[inner sep=1.0pt, circle, fill=black]  (v) at (-1.5,-1) {};
\node[below] at (v.south)  {$\ss v_1$};
\node[inner sep=1.0pt, circle, fill=black]  (w) at (1.5,-1) {};
\node[below] at (w.south)  {$\ss v_2$};
\draw[->] (u) to node[left,black] {$\ss e_1$}  (v);{}

\draw[->] (u) to node[right,black] {$\ss e_2$} (w);
\draw[->] (v) to [out=25, in=155] node[above,black] {$\ss f_1$}  (w);
\draw[->] (w) to [out=-155, in=-25]  node[below,black] {$\ss f_2$} (v);

\end{tikzpicture}
\]
The group $G={\mathbb Z}_2$ acts on $E$ by fixing $v$ and interchanging $v_1$ and $v_2$. This action takes the edge $e_1$ into $e_2$ and the edge $f_1$ into $f_2$. Since the set $\{v\}$ is hereditary and saturated, $C^*(E)$ has an ideal isomorphic to ${\mathcal K}$, the $C^*$-algebra of compact operators such that $C^*(E)/{\mathcal K}\cong M_2(C(\mathbb T))$.
Since $E$ has sources, the $K$-theory of $C^*(E)$ is computed using Theorem 3.2 in \cite{RS} and
\[K_0(C^*(E))= {\mathbb Z}\oplus{\mathbb Z}_2, \;\; K_1(C^*(E))=0.\]

We have $A=\mathbb C^3, {\mathcal H}=\mathbb C^4, A\rtimes {\mathbb Z}_2\cong \mathbb C\oplus \mathbb C\oplus M_2$ and ${\mathcal H}\rtimes {\mathbb Z}_2\cong {\mathbb C}^8$ which decomposes as ${\mathbb C}^2\oplus {\mathbb C}^2\oplus M_2$. The crossed product $C^*(E)\rtimes {\mathbb Z}_2$ is the $C^*$-algebra of the following graph of $C^*$-correspondences
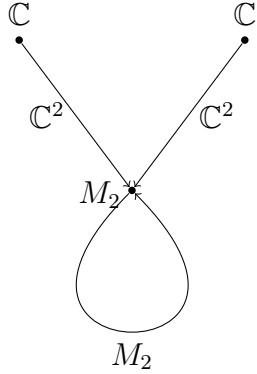
\[
\begin{tikzpicture}shorten >=0.4pt, >=stealth, semithick]
\renewcommand{\ss}{\scriptstyle}
\node[inner sep=1.0pt, circle, fill=black]  (u) at (-1.5,2) {};
\node[above] at (u.north)  {$ {\mathbb C}$};
\node[inner sep=1.0pt, circle, fill=black]  (v) at (1.5,2) {};
\node[above] at (v.north)  {${\mathbb C}$};
\node[inner sep=1.0pt, circle, fill=black]  (w) at (0,0) {};
\node[left] at (w.south)  {$M_2$};
\draw[->] (u) to  node[left,black] {$ {\mathbb C}^2$}  (w);{}

\draw[->] (v) to  node[right,black] {$ {\mathbb C}^2$} (w);
\draw[->] (w) ..controls (-2.5,-2.5) and (2.5,-2.5)..  (w) node[pos=0.5,below] {$ M_2$} ;

\end{tikzpicture}
\]
If $E'$ denotes the subjacent graph, there is an extension
\[0\to{\mathcal K}\to C^*(E')\to{\mathcal T}\to 0,\]
where ${\mathcal T}$ is the Toeplitz algebra. We have
\[K_0(C^*(E)\rtimes {\mathbb Z}_2)=K_0(C^*(E'))={\mathbb Z}^2,\;\; K_1(C^*(E)\rtimes {\mathbb Z}_2)=K_1(C^*(E'))=0.\]

\end{example}
\begin{example}
Let $S_3$ act on the graph $E$ with one vertex and three loops by permuting the loops. We get a $3$-dimensional representation $\rho$ of $S_3$ and a non-free action  on ${\mathcal O}_3$. We already know (see \cite{MRS}) that ${\mathcal O}_\rho\cong {\mathcal O}_3^{S_3}$ is a full corner in a graph algebra with incidence matrix
\[ \left[\begin{array}{ccc}1&0&1\\0&1&1\\1&1&2\end{array}\right].\]
Since $S_3\cong{\mathbb Z}_3\rtimes{\mathbb Z}_2$, we have ${\mathcal O}_3\rtimes S_3\cong({\mathcal O}_3\rtimes{\mathbb Z}_3)\rtimes{\mathbb Z}_2$. We will describe the graph of $C^*$-correspondences using this iterated crossed product.

It follows  that ${\mathcal O}_3\rtimes{\mathbb Z}_3$ is isomorphic to $C^*(E(c))$, where $c:E^1\to \widehat{{\mathbb Z}_3}$ is a cocycle and $E(c)$ is the  graph with three vertices $v_1,v_2, v_3$ and nine edges
 connecting each $v_i$ with $v_j$.
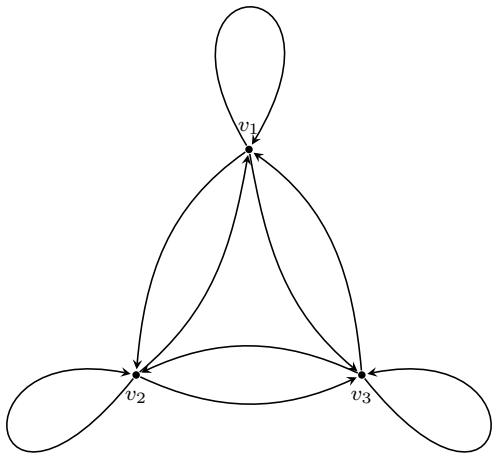
\[
\begin{tikzpicture}[shorten >=0.4pt, >=stealth, semithick]
\renewcommand{\ss}{\scriptstyle}
\node[inner sep=1.0pt, circle, fill=black]  (u) at (0,2) {};
\node[above] at (u.north)  {$\ss v_1$};
\node[inner sep=1.0pt, circle, fill=black]  (v) at (-1.5,-1) {};
\node[below] at (v.south)  {$\ss v_2$};
\node[inner sep=1.0pt, circle, fill=black]  (w) at (1.5,-1) {};
\node[below] at (w.south)  {$\ss v_3$};
\draw[->] (u) to [out=215, in=85] (v);{}
\draw[->] (v) to [out=40, in=-100] (u);
\draw[->] (w) to [out=95, in=-35] (u);
\draw[->] (u) to [out=-80, in=140] (w);
\draw[->] (v) to [out=-25, in=-155]  (w);
\draw[->] (w) to [out=155, in=25]  (v);
\draw[->] (u) ..controls (-1.5,4.5) and (1.5,4.5)..(u);
\draw[->] (v) ..controls (-3.5, -3.5) and (-4, -0.5) .. (v);
\draw[->] (w) ..controls (3.5,-3.5) and (4, -0.5).. (w);
\end{tikzpicture}
\]
The group ${\mathbb Z}_2$ acts on $E(c)$ by fixing $v_1$ and interchanging $v_2$ with $v_3$. The action on edges is uniquely determined. If $\pi$ is the corresponding representation of ${\mathbb Z}_2$  on ${\mathbb C}^9$, it follows that ${\mathcal O}_\pi\cong C^*(E(c))^{{\mathbb Z}_2}$. The quotient graph $E(c)/{\mathbb Z}_2$ has two vertices $u_1, u_2$ corresponding to the two orbits in $E(c)^0$ and five edges corresponding to the orbits in $E(c)^1$: one loop at $u_1$, two loops at $u_2$, one edge from $u_1$ to $u_2$ and one edge from $u_2$ to $u_1$.
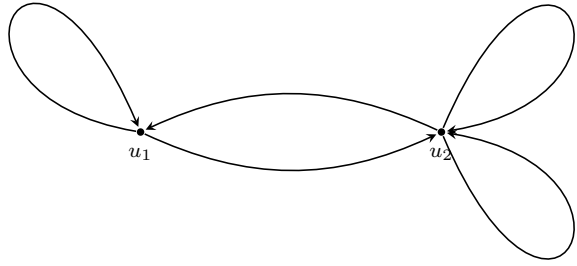
\[
\begin{tikzpicture}[shorten >=0.4pt,>=stealth, semithick]
\renewcommand{\ss}{\scriptstyle}
\node[inner sep=1.0pt, circle, fill=black]  (u) at (-2,0) {};
\node[below] at (u.south)  {$\ss u_1$};
\node[inner sep=1.0pt, circle, fill=black]  (v) at (2,0) {};
\node[below] at (v.south)  {$\ss u_2$};
\draw[->] (u) to [out=-25, in=-155]  (v);
\draw[->] (v) to [out=155, in=25]  (u);
\draw[->] (u) .. controls (-5,0.5) and (-3.5, 3.5) .. (u);
\draw[->] (v) ..controls (3.5,-3.5) and (5, -0.5).. (v);
\draw[->] (v) ..controls (3.5,3.5) and (5, 0.5).. (v);
\end{tikzpicture}
\]
We have $C^*(E(c))\cong {\mathcal O}_3$, so we get a non-free action of ${\mathbb Z}_2$ on ${\mathcal O}_3$. Here $A={\mathbb C}^3$ and $A\rtimes {\mathbb Z}_2\cong {\mathbb C}\oplus{\mathbb C}\oplus M_2\cong C^*(S_3)$.
 Moreover, ${\mathcal O}_3\rtimes S_3\cong C^*(E(c))\rtimes {\mathbb Z}_2$ is the $C^*$-algebra of the following graph of $C^*$-correspondences.
 For the vertex $u_1$ the  $C^*$-algebra is $C^*({\mathbb Z}_2)\cong {\mathbb C}^2$ and for $u_2$ the $C^*$-algebra is ${\mathbb C}^2\rtimes{\mathbb Z}_2\cong M_2$. For the loop at $u_1$ the $C^*$-correspondence is ${\mathbb C}^2$. For each of the two loops at $u_2$ we have a copy of $M_2$. Finally, for each of the remaining two edges in $E(c)/{\mathbb Z}_2$ we have ${\mathbb C}^2\oplus {\mathbb C}^2$. Since the vertex $u_1$ splits in two,
  the $C^*$-correspondence ${\mathbb C}^9\rtimes{\mathbb Z}_2$ over ${\mathbb C}^2\oplus M_2$ decomposes further as ${\mathbb C}\oplus {\mathbb C}\oplus M_2\oplus M_2\oplus{\mathbb C}^2\oplus{\mathbb C}^2\oplus{\mathbb C}^2\oplus{\mathbb C}^2$ and we get the following graph of minimal $C^*$-correspondences:
there are three vertices  with $C^*$-algebras ${\mathbb C}, {\mathbb C}$ and $M_2$ respectively. The incidence matrix of the graph is

\[\left[\begin{array}{ccc}1&0&1\\0&1&1\\1&1&2\end{array}\right]\]
 and the $C^*$-correspondences are as in the figure
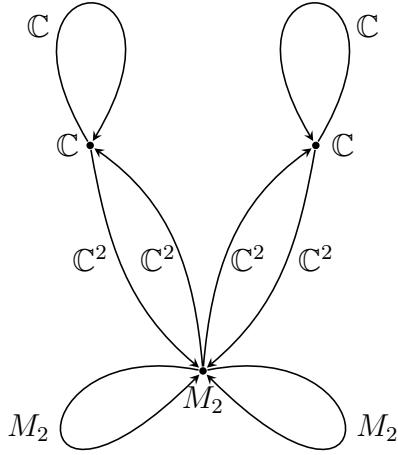
\[
\begin{tikzpicture}[shorten >=0.2pt,>=stealth, semithick]
\renewcommand{\ss}{\scriptstyle}
\node[inner sep=1.0pt, circle, fill=black]  (u) at (-1.5,2) {};
\node[left] at (u)  {$\mathbb C$};
\node[inner sep=1.0pt, circle, fill=black]  (v) at (1.5,2) {};
\node[right] at (v.east)  {$\mathbb C$};
\node[inner sep=1.0pt, circle, fill=black]  (w) at (0,-1) {};
\node[below] at (w.south)  {$ M_2$};
\draw[->] (w) to [out=95, in=-35] (u) node at (-1.5, 0.5)  {${\mathbb C}^2$};
\draw[->] (u) to [out=-80, in=140] (w) node at (-0.6, 0.5) {${\mathbb C}^2$};
\draw[->] (w) to [out=85, in=215]  (v) node at (0.6,0.5) {${\mathbb C}^2$};
\draw[->] (v) to [out=-100, in=40]  (w) node at (1.5, 0.5) {${\mathbb C}^2$};
\draw[->] (u) ..controls (-3,4.5) and (0,4.5)..(u) node[pos=0.3, left]{$\mathbb C$};
\draw[->] (v) ..controls (3,4.5) and (0,4.5)..(v) node[pos=0.3, right]{$\mathbb C$} ;
\draw[->] (w) ..controls (-2.5, -0.5) and (-2.5, -3.5) ..  (w)node[pos=0.5,left] {$ M_2$};
\draw[->] (w) ..controls (2.5,-0.5) and (2.5, -3.5).. (w)node[pos=0.5,right] {$M_2$} ;
\end{tikzpicture}
\]
The group $S_3$ also acts on the core $M_{3^\infty}$ of ${\mathcal O}_3$. Since ${\mathcal O}_3\rtimes S_3$ is strongly Morita equivalent to a graph algebra, it follows that its core is strongly Morita equivalent to $M_{3^\infty}\rtimes S_3$.
\end{example}

\begin{example}

Let the dihedral group $D_4\cong {\mathbb Z}_4\rtimes{\mathbb Z}_2$ act in the natural way on the ``cross" graph $E$ with five vertices and four edges by fixing the central vertex.
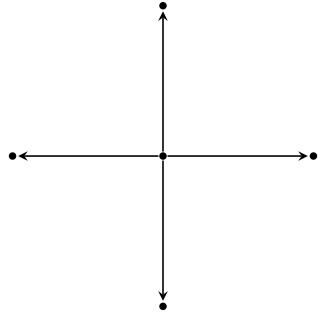
\[
\begin{tikzpicture}[shorten >=0.4pt,>=stealth, semithick]
\node[inner sep=1.0pt, circle, fill=black]  (u) at (0,0) {};
\node[inner sep=1.0pt, circle, fill=black]  (v) at (-2,0) {};
\node[inner sep=1.0pt, circle, fill=black]  (w) at (2,0) {};
\node[inner sep=1.0pt, circle, fill=black]  (x) at (0,2) {};
\node[inner sep=1.0pt, circle, fill=black]  (y) at (0,-2) {};
\draw[->] (u) to  (v) {} ;
\draw[->] (u) to  (w) {} ;

\draw[->] (u) to  (x) {} ;

\draw[->] (u) to  (y) {} ;

\end{tikzpicture}
\]
We get an action of $D_4$ on $C^*(E)\cong M_5$. It is known that $C^*(E)\rtimes D_4\cong M_5\otimes C^*(D_4)$. The vertex space has two orbits and the edge space one orbit.
We have ${\mathbb C}^5\rtimes D_4\cong C^*(D_4)\oplus {\mathbb C}^4\rtimes D_4\cong {\mathbb C}^4\oplus M_2\oplus M_4\oplus M_4$ and ${\mathbb C}^4\rtimes D_4$ decomposes as ${\mathbb C}^4\oplus{\mathbb C}^4\oplus {\mathbb C}^4\oplus {\mathbb C}^4\oplus M_{2,4}\oplus M_{2,4}$, obtaining the following graph of $C^*$-correspondences:
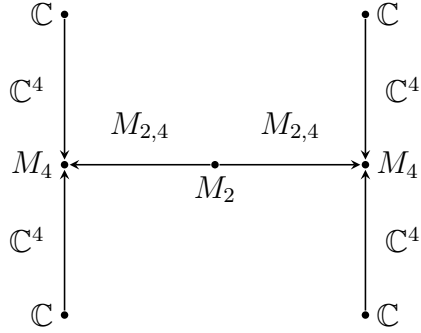
\[
\begin{tikzpicture}[shorten >=0.2pt,>=stealth, semithick]
\node[inner sep=1.0pt, circle, fill=black]  (u) at (0,0) {};
\node[below] at (u)  {$M_2$};
\node[inner sep=1.0pt, circle, fill=black]  (v) at (-2,0) {};
\node[left] at (v)  {$M_4$};
\node[inner sep=1.0pt, circle, fill=black]  (w) at (2,0) {};
\node[right] at (w)  {$M_4$};
\node[inner sep=1.0pt, circle, fill=black]  (x) at (-2,2) {};
\node[left] at (x)  {$\mathbb C$};
\node[inner sep=1.0pt, circle, fill=black]  (y) at (-2,-2) {};
\node[left] at (y)  {$\mathbb C$};
\node[inner sep=1.0pt, circle, fill=black]  (t) at (2,2) {};
\node[right] at (t)  {$\mathbb C$};
\node[inner sep=1.0pt, circle, fill=black]  (z) at (2,-2) {};
\node[right] at (z)  {$\mathbb C$};
\draw[->] (u) to  (v)node at (-1,0.5) {$M_{2,4}$} ;
\draw[->] (u) to  (w) node at (1,0.5) {$M_{2,4}$} ;

\draw[->] (x) to  (v) node at (-2.5,1) {${\mathbb C}^4$} ;

\draw[->] (y) to  (v) node at (-2.5,-1) {${\mathbb C}^4$} ;

\draw[->] (t) to (w) node at (2.5,1) {${\mathbb C}^4$} ;
\draw[->] (z) to (w) node at (2.5,-1) {${\mathbb C}^4$} ;
\end{tikzpicture}
\]
\end{example}

\begin{example}{\label{sym}} Let the symmetric group $S_3=\la\tau, \sigma\ra$ act on ${\mathcal O}_2$ by $\tau(s_1)=s_2, \tau(s_2)=s_1, \sigma(s_1)=\omega s_1, \sigma(s_2)=\omega^2 s_2$, where $\tau=(12), \sigma=(123)$ and $\omega^2+\omega+1=0$. This action corresponds to the two-dimensional irreducible representation $\rho$ of $S_3$, and it commutes with the gauge action. Then ${\mathcal O}_2^{S_3}$ and ${\mathcal O}_2\rtimes S_3$ have the K-theory of ${\mathcal O}_3$, since (see \cite{MRS}) ${\mathcal O}_\rho$ is a full corner in the  algebra of the graph with three vertices and incidence matrix
\[B=\left[\begin{array}{ccc}0&1&0\\1&1&1\\0&1&0\end{array}\right].\]
In particular, we get an action of $S_3$ on the CAR algebra $M_{2^\infty}$ and
\[K_0(M_{2^\infty}\rtimes S_3)\cong \varinjlim({\mathbb Z}^3, B).\]
\end{example}

{\small \hfill Received March 26, 2015}

\end{document}